\documentclass[11pt]{article}   
\usepackage[greek,english]{babel}  
\usepackage{graphicx}
\usepackage[pdftex,bookmarks]{hyperref}
\usepackage{amsmath}
 \usepackage{amssymb}
\usepackage{amsmath,amsthm}
\usepackage{verbatim}
\newtheorem{theorem}{Theorem}[section]

\newtheorem{remark}[theorem]{Remark}
\newtheorem{proposition}[theorem]{Proposition}

\newcommand{\eps}{\varepsilon}

\titlepage 
\title{Vanishing diffusion limits for planar fronts \\ in bistable models with saturation}
\author{Maurizio Garrione}
\date{\small{Dipartimento di Matematica, Politecnico di Milano \\ 
Piazza Leonardo da Vinci 32, 20133 Milano \\
e-mail address: \tt{maurizio.garrione@polimi.it}}}

\begin{document}
\maketitle

\begin{abstract}
We deal with heteroclinic planar fronts for parameter-dependent reaction-diffusion equations with bistable reaction and saturating diffusive term like 
\begin{equation}\label{eqab}
u_t=\eps \, \textnormal{div}\, \left(\frac{\nabla u}{\sqrt{1+\vert \nabla u \vert^2}}\right) + f(u), \quad u=u(x, t), \; x \in \mathbb{R}^n, \, t \in \mathbb{R},
\end{equation}
analyzing in particular their behavior for $\eps \to 0$. First, we construct monotone and non-monotone planar traveling waves, using a change of variables allowing to analyze a two-point problem for a suitable first-order reduction of \eqref{eqab}; then, we investigate their asymptotic behavior for $\eps \to 0$, showing in particular that the convergence of the critical fronts to a suitable step function may occur passing through discontinuous solutions. 
\end{abstract}

\smallbreak

\section{Introduction}

In this paper, we analyze the behavior of traveling-wave type solutions for a reaction-diffusion model with saturating diffusive term and bistable reaction in dependence on a small parameter $\eps$, showing that the presence of the saturation induces singularities for $\eps$ sufficiently small and thus the limit procedure for $\eps \to 0$ necessarily passes through discontinuous steady states.
\smallbreak
In the analysis of reaction-diffusion models, traveling waves are probably the first pattern to be studied, since despite their quite simple form they often manage to give some useful insight on the dynamics. Their appearance is essentially due to the interplay between the spatial action exerted by the diffusion and the constructive/destructive effect due to the reaction, which contribute to create one (or more) profile propagating with a fixed speed $c$ and having parallel planes as level surfaces. 
\smallbreak
Usually, it is assumed that both $u\equiv 0$ and $u \equiv 1$ are equilibrium solutions and such profiles take values between $0$ and $1$, as in \cite{Lut} and in the celebrated paper by Fisher \cite{Fis}, where the unknown $u$ indicated in fact the relative concentration of a gene undergoing an advantageous mutation in a population. In this case, the natural waves to be studied are the ones leading to a complete spread of the advantageous gene starting with $0$-concentration, namely the ones connecting the equilibrium states $0$ and $1$. Both the shapes of the reaction and of the diffusion term play a primary role for the existence of such solutions, and may give rise to several different portraits.
\smallbreak
As for the former, for the linear model 
\begin{equation}\label{linintro}
u_t= \Delta u + f(u), \quad u=u(x, t), \; x \in \mathbb{R}^n, \, t \in \mathbb{R},
\end{equation}
it has been proved (see, e.g., \cite{AroWei, KolPetPis}) that if the reaction is always positive and linearly controlled near $0$ and $1$ then there exists a (monotone, cf. \cite{FifM}, and regular) heteroclinic profile between $0$ and $1$ for any (positive) speed starting from a suitable value $c^*$, named \emph{critical speed}. On the other hand, if the reaction is non-negative in a neighborhood of $0$, then there exists a unique admissible speed for (regular and monotone) heteroclinic profiles between $0$ and $1$. This means that reactions helping the time growth of the substance at any level of concentration give rise to multiplicity of admissible transitions, while if the reaction obstructs the growth of $u$ penalizing low concentrations (this being also called \emph{strong Allee effect} in biological models), then the speed of the profile is the result of a careful balance between the strengths of the diffusion and of the reaction.
\smallbreak
The role of diffusions other than the linear one has instead been studied more recently and has given rise to a huge number of works in many different directions. Our focus is here on the so-called \emph{strongly saturating diffusions}, whose study was introduced by Rosenau and co-workers \cite{KurRos, Ros}. A motivation for considering such diffusions lies in the need to restore the finiteness of the energy along sharp interfaces, in order to possibly admit the existence of discontinuous solutions. A careful study performed in \cite{Ros} (see also \cite{KurRos}) for the $1$-dimensional case revealed that, for such a feature to hold, the diffusive term has to be of the kind $(P(u_x))_x$, where $P$ is a bounded increasing $C^1$-function satisfying
\begin{equation}\label{diff}
\int_0^{+\infty} s P'(s) \, ds < +\infty. 
\end{equation}
As a basic example, one could choose
$$
P(s)=\frac{s}{\sqrt{1+s^2}},
$$
corresponding to the so-called \emph{mean curvature operator} (in the Euclidean space). In dimension $n$, a natural extension can be obtained selecting diffusions having the form 
$
\textnormal{div}\,(\varphi(\vert \nabla u \vert) \nabla u), 
$  
where $\varphi$ is of class $C^1$, so that \eqref{diff} takes the form 
\begin{equation}\label{diffn}
\int_0^{+\infty} [s\varphi(s) + s^2\varphi'(s)] \, ds < +\infty,
\end{equation}
as we will see later on.
We will briefly comment about the consequences of assumption \eqref{diffn} on the features of planar fronts in Remark \ref{singolare}.
The conditions of strong saturation \eqref{diff} and \eqref{diffn} are opposed to the ones characterizing weakly saturating diffusions, reading as
\begin{equation}\label{weaksat}
\int_0^{+\infty} s P'(s) \, ds = +\infty \quad \textrm{ and } \quad \int_0^{+\infty} [s\varphi(s)+s^2\varphi'(s)] \, ds = +\infty. 
\end{equation}
While, for the solutions of the general PDE (possibly with convection), both weakly and strongly saturating diffusions are accompanied by the natural emergence of singularities (see, e.g., \cite{GooKurRos}), speaking about planar fronts it is the rate of saturation which plays a crucial role for the appearance of discontinuous solutions: in case \eqref{weaksat} is fulfilled, there are no deep qualitative differences in the dynamics of fronts with respect to the linear model \eqref{linintro}, as highlighted in Remark \ref{singolare}.
On the contrary, strongly saturating diffusions \eqref{diff} (or \eqref{diffn}) may create discontinuous stationary waves. As shown in \cite{GarSan, KurRos}, though condition \eqref{diff} reveals outcomes that are  analogous to the ones of the linear diffusion case for \emph{monostable} reactions (the aforementioned positive ones), in fact discontinuous steady states may appear if the reaction is of \emph{bistable} type, i.e.,  
if $f$ is negative in a right neighborhood of $0$ and changes sign in correspondence of a third equilibrium $\alpha \in (0, 1)$, see assumption \textbf{(H2)} in the next section. This occurs when the negative part of the reaction is so strong (in $L^1$-norm) that the weakness of the saturating diffusion is not able to counterbalance it giving rise to a regular profile. 
\smallbreak
This suggests that a loss of regularity should always appear in the vanishing diffusion limit of planar fronts for the problem 
\begin{equation}\label{saturatointro}
u_t= \eps \textnormal{div}\,(\varphi(\vert \nabla u \vert)\nabla u) + f(u), \quad u=u(x, t), \; x \in \mathbb{R}^n, \, t \in \mathbb{R}, 
\end{equation}
for any bistable reaction term $f$;
indeed, the smallness of the parameter $\eps$ makes the weakness of the saturating diffusion arbitrarily accentuated. Incidentally, recall that 
the study of vanishing viscosity limits is a typical procedure in hyperbolic dynamics: for instance, adding a small regularizing diffusive term in reaction-convection problems allows to recover entropy solutions as the limit of regular solutions for $\eps \to 0$ (see, e.g., \cite{Cro, CroMas, MatMer}). 
\smallbreak
In this paper, we thus deal with the limit of heteroclinic planar traveling waves of equation \eqref{saturatointro} for $\eps \to 0$, investigating both the behavior of the wave speeds and the shape of the limit profile. Assuming that 
$$
f(0)=f(\alpha)=f(1)=0, \qquad f(s) < 0 \textrm{ for } s \in (0, \alpha), \qquad f(s) > 0 \textrm{ for } s \in (\alpha, 1),
$$
for some $\alpha \in (0, 1)$, here we have to distinguish between \emph{monostable-type} fronts, namely fronts connecting $\alpha$ and $1$ (or $0$ and $\alpha$), and \emph{bistable-type} ones, connecting $0$ and $1$. As we will see in Section \ref{sez3}, for the first ones the picture is very similar to the one for the linear case, for which we refer, e.g., to \cite{HilKim}; see also \cite{AriCamMar, MalMarMat}. What is new for \emph{bistable-type} fronts is instead that the diffusion process 
\begin{center}
\emph{slows down on decreasing of $\eps$ until it occurs with $0$-speed \\ already for a critical positive value $\bar{\eps}$},
\end{center}
given by $\bar{\eps}=\int_0^1 f^-(s) \, ds$ (where $f^-(s)=\max\{-f, 0\}$). Then, the convergence for $\eps \to 0$ can only occur passing through discontinuous steady states (see Theorems \ref{velconv0} and \ref{convAE} below).
\smallbreak
In the proof of the above results, we exploit a suitable change of variables reducing the order of the problem, strictly related to the one introduced in \cite{MalMar} - see also \cite{GarSan} - and used in two slightly different versions, see Remark \ref{altroc} below. The core of the method then relies both in a shooting technique making use of lower and upper solutions arguments and in a direct convergence analysis deeply exploiting the properties of the mean curvature operator (in Remark \ref{remMink}, we make a brief comparison with the qualitatively different case of the so-called \emph{Minkowski curvature} or \emph{Born-Infeld} operator). This technique allows us to identify in a quite precise way the development of singularities and to make general considerations about nonmonotone heteroclinic traveling waves as well, as shown in Section \ref{sez2}.

\section{Construction of the traveling waves}\label{sez2}

Our focus is on the partial differential equation 
\begin{equation}\label{laPDE}
u_t=\eps \, \textnormal{div}\, (\varphi(\vert \nabla u \vert)\nabla u) + f(u), \quad u=u(x, t), \; x \in \mathbb{R}^n, \, t \in \mathbb{R},
\end{equation} 
where $\varphi:[0, +\infty) \to \mathbb{R}$, $\eps$ is a small parameter ($\eps \to 0$) and $f$ is a bistable reaction term. Precisely, we will assume the following:
\smallbreak
\noindent
\textbf{(H1)} $P(s):=s\varphi(\vert s \vert)$ is a $C^1$-function fulfilling \eqref{diff}, and $P'(s) > 0$ for every $s \in \mathbb{R}$;
\smallbreak
\noindent
\textbf{(H2)} $f: [0, 1] \to \mathbb{R}$ is continuous, $f(0)=0=f(1)$, there exists $\alpha \in (0, 1)$ such that $f(\alpha) = 0$ and $f(s)(s-\alpha) > 0$ for every $s \in (0, 1) \setminus \{0, \alpha, 1\}$. Moreover, $\int_0^1 f(s) \, ds > 0$, there exists $f'(\alpha) > 0$ and
\begin{equation}\label{ipof}
\vert f(s) \vert \leq f'(\alpha) \vert s-\alpha \vert \qquad \textrm{ for every } s \in [0, 1]. 
\end{equation}
Finally, there exists $l > 0$ such that $f(s) \geq -ls$, $f(s) \leq l(1-s)$ for every $s \in [0, 1]$. 
\smallbreak
As already mentioned in the Introduction, our aim is to study heteroclinic planar traveling waves for \eqref{laPDE}, namely solutions of the kind $u(x, t)=v(x \cdot e + ct)$ connecting two equilibria, where $e$ is a fixed vector on the unit sphere of $\mathbb{R}^n$.
Thanks to the fact that $\vert e \vert = 1$, replacing $u(x, t)=v(x\cdot e + ct)$ into \eqref{laPDE} provides 
\begin{equation}\label{TWgeneral1}
\eps(\varphi(\vert v' \vert)v')' - c v' + f(v) =0,
\end{equation}
which recalling the definition of $P$ and dividing both sides by $\eps$ can be rewritten as 
\begin{equation}\label{TWgeneral}
(P(v'))' - b_\eps v' + g_\eps(v) =0,
\end{equation}
where $b_\eps=c/\eps$ and $g_\eps(s)=f(s)/\eps$. In \eqref{TWgeneral1} and \eqref{TWgeneral}, we denote by $z$ the independent variable; moreover, we will highlight the dependence on the parameter $\eps$ by writing $c=c_\eps$. 

Let us briefly comment about hypotheses \textbf{(H1)}-\textbf{(H2)}. In analogy with \cite{Ros}, dealing with the $1$-dimensional case, a diffusive term fulfilling  \textbf{(H1)} will be called \emph{strongly saturating}; as we mention in Remark \ref{singolare}, such a terminology is justified by the appearance of a singularity resulting into the existence of discontinuous planar stationary waves, similarly as in the $1$-dimensional case. Notice that \textbf{(H1)} implies, in particular, that $P$ is bounded and
$
\lim_{s \to +\infty} \varphi(s) = 0;
$
moreover, the fact that $P' > 0$ ensures that the differential equation \eqref{TWgeneral} is nondegenerate.
As for \textbf{(H2)}, the sign condition on $f$ is referred to by saying that $f$ is a \emph{bistable reaction term}, and fixes the sign for the speeds of increasing traveling waves connecting $\alpha$ and $1$ (resp., $0$ and $\alpha$), which can only be positive (resp., negative); 
the assumption $\int_0^1 f(s) \, ds > 0$ fixes instead the sign for the speeds of increasing traveling waves connecting $0$ and $1$, which can only be positive (just integrate \eqref{TWgeneral1} on the whole real line and exploit the fact that $v'(-\infty)=v'(+\infty)=0$); the regularity condition \eqref{ipof} in $\alpha$ is required in order to have neat estimates of the critical speeds - see the observation after the proof of Proposition \ref{monotoni} - and, together with the linear controls near the equilibria $0$ and $1$, provides a complete equivalence between heteroclinic solutions of \eqref{TWgeneral} and solutions of the associated first-order reduction, see also the comment before Proposition \ref{monotoni}. 

Given now $q_1, q_2 \in \{0, \alpha, 1\}$ such that $f(q_1)=f(q_2)=0$, $q_1 < q_2$, we first observe that if $v$ is a solution of \eqref{TWgeneral} with $v(-\infty)=q_1$, $v(+\infty)=q_2$, the fact that \eqref{TWgeneral} is autonomous implies that also $z \mapsto v(z+\tau)$ solves \eqref{TWgeneral} and connects the same two equilibria at infinity, for every $\tau \in \mathbb{R}$. A whole one-parameter family of heteroclinics propagating with the same speed $c$ is thus automatically found. In case $z \mapsto v(z)$ is \emph{monotone} - thus being called a heteroclinic \emph{front} - we can recover the uniqueness for fixed speed $c$ by imposing
$$
v(0)=\frac{q_1+q_2}{2}.
$$
Moreover, for monotone solutions it is sufficient to focus on \emph{nonnegative} speeds: 
if $v$ is an increasing front having speed $b_\eps > 0$ and connecting $q_1 < q_2$, $w(z)=v(-z)$ solves
$$
(\varphi(\vert w' \vert)w')' + b_\eps w' + g_\eps(w) = 0,  
$$
is decreasing and connects $q_2$ with $q_1$, so it suffices to change the sign of the speed to have a front with opposite monotonicity.

Let us thus focus first on fronts. The procedure we use here is a first-order reduction making use of the change of variables exploited in \cite{GarSan, GarStr, MalMar}. 
Namely, the monotonicity of $v$ allows us to write $z=z(v)$ - where $v$ becomes the new independent variable - and obtain a first order differential equation for $\phi(v)=v'(z(v))$. By doing explicit computations, for which we refer also to \cite{GarSan}, this yields 
$$
\frac{d}{dv} Q(\phi(v)) - b_\eps \phi(v) + g_\eps(v) = 0,
$$
where $s \mapsto Q(s)$ is the primitive of $sP'(s)=s\varphi(s)+s^2\varphi'(s)$ satisfying $Q(0)=0$. Focusing for the moment on the case of increasing fronts (decreasing ones will be obtained by a change of variable, see the proof of Proposition \ref{monotoni}), we set $y(v)=Q(\phi(v))$ and notice that \eqref{diff} implies that the range of $Q$ is a bounded interval, say $[0, M_0)$. We thus obtain
\begin{equation}\label{ilsistemag}
\left\{
\begin{array}{l}
\displaystyle y'= b_\eps R(y) - g_\eps(v), \vspace{0.1cm}\\
y(q_1)=0, \, y(q_2)=0, \, 0 < y(v) < M_0 \textrm{ for } v \in (q_1, q_2),
\end{array}
\right.
\end{equation}
where $R$ denotes the functional inverse of $Q$, which is well defined because $(0, +\infty) \ni s \mapsto sP'(s)$ is positive in view of assumption \textbf{(H1)}. The fact that $Q([0, +\infty))=[0, M_0)$ implies that $R$ possesses a singularity, and this will turn into the possible existence of discontinuous stationary solutions (see also \cite[Section 2]{GarStr} and Remark \ref{singolare}). Obviously, the two boundary conditions in \eqref{ilsistemag} come from the fact that a monotone heteroclinic profile connecting $q_1$ and $q_2$ has to have zero derivative in correspondence of such equilibria.  As a notation, when dealing with systems having similar form to \eqref{ilsistemag}, we write explicitly the independent variable $v$ only in the expression of the reaction term. 

To make things more readable, henceforth we make the precise choice 
\begin{equation}\label{lascelta}
\varphi(s)=\frac{1}{\sqrt{1+s^2}} \quad \Rightarrow \quad P(s)=\frac{s}{\sqrt{1+s^2}},
\end{equation}
for which we can perform more explicit computations; however, our framework guarantees that, up to modifying the bounds on the critical speeds which will be given from here on - which essentially depend on the behavior of $R$ in $0$ - the results which follow can be stated in the same way also for the general equation \eqref{laPDE}  (see for instance the remarks after the convergence statements in Section \ref{sez3}). Under assumption \eqref{lascelta}, \eqref{TWgeneral1} and \eqref{TWgeneral} read respectively as
\begin{equation}\label{TW}
\eps\left(\frac{v'}{\sqrt{1+(v')^2}}\right)' - c v' + f(v) = 0
\end{equation}
and 
\begin{equation}\label{TWW}
\left(\frac{v'}{\sqrt{1+(v')^2}}\right)' - b_\eps v' + g_\eps(v) = 0,
\end{equation}
and we have 
$
R(s)=\frac{\sqrt{s(2-s)}}{1-s}.
$ 
Hence, \eqref{ilsistemag} becomes
\begin{equation}\label{ilsistema}
\left\{
\begin{array}{l}
\displaystyle y'= b_\eps \frac{\sqrt{y(2-y)}}{1-y} - g_\eps(v), \vspace{0.1cm}\\
y(q_1)=0, \, y(q_2)=0, \, 0 < y(v) < 1 \textrm{ for } v \in (q_1, q_2),
\end{array}
\right.
\end{equation}
where
\begin{equation}\label{cambio1}
y(v)=
\left(1 - \frac{1}{\sqrt{1+v'(z(v))^2}}\right).
\end{equation}
In view of assumption \textbf{(H2)}, once a solution $y$ of \eqref{ilsistema} is found, it is possible to recover a planar front connecting $q_1$ and $q_2$ by solving the Cauchy problem 
$$
\left\{
\begin{array}{l}
v'(z)= \displaystyle \frac{\sqrt{y(v(z))(2-y(v(z)))}}{1-y(v(z))} \vspace{0.1cm} \\
\displaystyle v(0)=\frac{q_1+q_2}{2}
\end{array}
\right.
$$ 
(cf. \cite[Remark 2.1 and Remark 2.2]{GarSan}).
Thanks to \cite[Proposition 3.2 and Proposition 3.9]{GarSan}, we can then characterize the fronts in dependence on $c$ as follows. 
\begin{proposition}\label{monotoni}
Fix $\eps > 0$ and assume \textbf{(H2)}. Moreover, set 
\begin{equation}\label{velcritica}
b_\eps^+=2\sqrt{g_\eps'(\alpha)}=2\sqrt{\frac{f'(\alpha)}{\eps}}.
\end{equation}
Then, there exist $0 < b_\eps^* < b_\eps^+$ such that regular heteroclinic fronts for \eqref{TWW} exist if and only if 
\begin{itemize}
\item $q_1=\alpha$, $q_2=1$, $b \geq b_\eps^+$ (increasing);
\item $q_1=0$, $q_2=\alpha$, $b \geq b_\eps^+$ (decreasing); 
\item $q_1=0$, $q_2=1$, $\eps > \int_0^1 f^-(s) \, ds$, $b=b_\eps^*$ (increasing). 
\end{itemize}
\end{proposition}
\begin{proof}
The first item follows from the fact that $f\vert_{[\alpha, 1]}$ is positive and fulfills \eqref{ipof}, so that \cite[Proposition 3.2]{GarSan} finds application. Indeed, estimate \eqref{ipof} implies that $g_\eps$ in \eqref{ilsistema} satisfies the control
\begin{equation}\label{controllog1}
g_\eps(s) \leq \frac{M(s-\alpha)}{\sqrt{1-\min\{M, 1\}(s-\alpha)^2}}
\end{equation}
for $s > \alpha$, being $M=f'(\alpha)/\eps$, so that for any $b_\eps \geq 2\sqrt{f'(\alpha)/\eps}=2\sqrt{g_\eps'(\alpha)}$ there exists an increasing front connecting $\alpha$ and $1$ having speed $b_\eps$.

As for the second item, in view of \cite[Remark 2.3]{GarSan} a decreasing front from $\alpha$ to $0$ having speed $b_\eps$ corresponds to a solution $y$ of 
$$
\left\{
\begin{array}{l}
\displaystyle y'= -b_\eps \frac{\sqrt{y(2-y)}}{1-y} - g_\eps(v), \vspace{0.1cm}\\
y(0)=0, \, y(\alpha)=0, \, 0 < y(v) < 1 \textrm{ for } v \in (0, \alpha);
\end{array}
\right.
$$
setting $Y(v)=y(1-v)$,
we have that $Y$ satisfies 
$$
\left\{
\begin{array}{l}
\displaystyle Y'= b_\eps \frac{\sqrt{Y(2-Y)}}{1-Y} - k_\eps(v), \vspace{0.1cm}\\
Y(1)=0, \, Y(1-\alpha)=0, \, 0 < Y(v) < 1 \textrm{ for } v \in (1-\alpha, 1),
\end{array}
\right.
$$
with $k_\eps(v)=-g_\eps(1-v)$. Since the critical speed for such a problem is given by $2\sqrt{k_\eps'(1-\alpha)}$ $=$ $2\sqrt{g_\eps'(\alpha)}=b_\eps^+$ \cite[Proposition 3.2]{GarSan} (notice that $k_\eps$ satisfies the control required therein in the point $1-\alpha$), we deduce that decreasing fronts connecting $0$ and $\alpha$ exist for $b \geq b_\eps^+$.  

Finally, the third item is a direct consequence of \cite[Proposition 3.9]{GarSan}, after noticing that the assumption $\eps > \int_0^1 f^-(s) \, ds$ corresponds to $\int_0^1 g^-(s) \, ds < 1$, which is actually the condition needed therein in order to have regular fronts connecting $0$ and $1$.
\end{proof}
Notice that assumption \eqref{ipof} allows us to explicitly write the critical speed relative to fronts connecting $\alpha$ and $1$, being given by \eqref{velcritica},
so that $c_\eps^+ = \eps b_\eps^+= 2\sqrt{\eps f'(\alpha)}$ (as also observed in \cite{GarStr}). Thus, the more $\eps$ approaches $0$, the more the original speeds corresponding to monotone fronts tend to invade the whole real line.
\begin{remark}\label{stimaM}
\textnormal{
If \eqref{ipof} is not fulfilled, one can state a similar result but with a less precise bound on the value of $b_\eps^+$. On the one hand, looking at the dynamics near the equilibrium $\alpha$ one sees that it is always true that $b_\eps^+ \geq 2\sqrt{g_\eps'(\alpha)}$ (cf., e.g., \cite{CoeSan, GarSan}); on the other hand, as already remarked in the previous proof, it is proved in \cite{GarSan} that if there exists $M > 0$ so that \eqref{controllog1} is fulfilled, then $b_\eps^+ \leq 2\sqrt{M}$. Notice that if the given reaction term $f$ satisfies 
\begin{equation}\label{stimaGS}
f(s) \leq \frac{\widetilde{M}(s-\alpha)}{\sqrt{1-\min\{\widetilde{M}, 1\}(s-\alpha)^2}}
\end{equation}
for $s > \alpha$ then $g_\eps=f/\eps$ satisfies \eqref{controllog1} with $M'=\widetilde{M}/\eps$; therefore, the upper bound for the critical speed $c_\eps^+$ relative to monostable-type fronts for \eqref{TW} rescales as well with $\sqrt{\eps}$ when $f$ is replaced by $g_\eps=f/\eps$.}
\end{remark}
The third item in Proposition \ref{monotoni} means that the evolution is not able to support regular fronts if the strength of the negative part of $f$ is too high, as previously mentioned, because the solutions of \eqref{ilsistema} blow-up to the boundary $\{y=1\}$. Indeed, if $\eps$ is too small we have the following. 
\begin{proposition}\label{dss}
Fix $\eps \leq \int_0^1 f^-(s) \, ds$. Then, the only heteroclinic front connecting $0$ and $1$ for \eqref{TWW} is a discontinuous steady state. 
\end{proposition}
\begin{proof}
Since $\int_0^1 f(s) \, ds > 0$, the statement follows from \cite[Proposition 5.2]{GarSan}, after having noticed that planar traveling wave-type solutions for \eqref{laPDE} having nonzero speed cannot be discontinuous 
(see \cite{KurRos} and the proof of Theorem \ref{convinv} below). 
\end{proof}
As for the statement of Proposition \ref{dss}, we mean the solution in weak-$BV_{\textnormal{loc}}$ sense, as in \cite[Example 1.1 and Definition 1.1]{BonObeOma}. Here, in order to recover uniqueness, we impose that the (only) discontinuity of the solution occurs in correspondence of $z=0$. We also arbitrarily continue to set $v(0)=1/2$, even if $v$ is therein not defined. Notice that if it were $\int_0^1 f(s) \, ds = 0$, then for $\eps > \int_0^1 f^-(s) \, ds$ the heteroclinic connection between $0$ and $1$ would be a regular front with zero speed; for $\eps = \int_0^1 f^-(s) \, ds$, such a profile would have infinite derivative when taking the value $1/2$, namely for $z=0$ (though remaining continuous therein; such a kind of solution was called \emph{border steady state} in \cite{GarSan}). However, the results of Section \ref{sez3} in terms of convergence for $\eps \to 0$ would hold all the same also in this case. The $BV$-setting is actually the most natural for the study of problems ruled by the curvature operator (see, e.g., \cite{BonObe}), especially on bounded intervals, and the transition between regular and $BV$-solutions is a quite typical phenomenon in this kind of quasilinear problems, see for instance \cite{ BonObeOma, BurGri, LopOmaRiv}. 
\begin{remark}\label{singolare}
\textnormal{Proposition \ref{dss} highlights a phenomenon which is peculiar of strongly saturating diffusions: indeed, it is condition \eqref{diff} which ensures that $R$ appearing in \eqref{ilsistemag} possesses a singularity, so that the domain of the right-hand side of the differential equation in \eqref{ilsistemag} cannot be the whole half-plane $y > 0$. The existence of discontinuous steady states thus arises and in the first-order model such solutions are simply found whenever both the (respectively, forward and backward) solutions of 
$$
\left\{
\begin{array}{l}
\displaystyle y'= b_\eps R(y) - g_\eps(v) \\
y(q_1)=0
\end{array}
\right.
\quad 
\textrm{ and }
\quad 
\left\{
\begin{array}{l}
\displaystyle y'= b_\eps R(y) - g_\eps(v) \\
y(q_2)=0
\end{array}
\right.
$$
blow up to the boundary $\{y=1\}$. 
Among the general saturating diffusions entering our setting, one could consider for instance
\begin{equation}\label{equazione2}
P(s)=\frac{s^m}{\sqrt{1+\delta s^{2m}}}, 
\end{equation}
for fixed $m > 1$ and $\delta > 0$. 
Here the singularity for $R$ occurs in correspondence of 
$y=\int_0^{+\infty} sP'(s) \, ds < +\infty$,
as observed, e.g., in \cite[Remark 2.1]{GarStr}. 
Similarly, the arguments in the present paper can be adapted to the case of a density-dependent diffusion of the kind $\textnormal{div}\,(\varphi(\vert \nabla D(u) \vert) \nabla D(u))$, with $D$ a strictly increasing function, as studied in \cite{GarStr}.
On the other hand, if \eqref{weaksat} holds then $R$ possesses no singularities and it is possible to find a regular front for any $\eps > 0$ by simply mimicking the technique used in the linear case \cite{BonSan} (exploiting the uniqueness for the backward Cauchy problem centered in $q_2$, here holding thanks to the monotonicity properties of $R$).}
\end{remark}

As a notation, henceforth let us set 
\begin{equation}\label{notazione}
\textrm{(PB)}_{b_\eps, q}^+:= \;
\left\{
\begin{array}{l}
\displaystyle y'= b_\eps \frac{\sqrt{y(2-y)}}{1-y} - g_\eps(v) \vspace{0.1cm}\\
y(q)=0, \, v > q
\end{array}
\right.
\quad 
\textrm{(PB)}_{b_\eps, q}^-:= \;
\left\{
\begin{array}{l}
\displaystyle y'= b_\eps \frac{\sqrt{y(2-y)}}{1-y} - g_\eps(v) \vspace{0.1cm}\\
y(q)=0, \, v < q,
\end{array}
\right.
\end{equation}
explicitly highlighting the dependence of the considered problem on the speed $b_\eps$, on the point $q$ where the initial condition is given and on the fact that it is solved forward or backward (through the superscripts $+$ and $-$, respectively). Correspondingly, let us denote the solutions to the problems in \eqref{notazione} as $y_{b_\eps, q}^+$ and $y_{b_\eps, q}^-$, respectively.  
One of the key points to obtain the increasing fronts of Proposition \ref{monotoni} is to study the qualitative properties of $y_{b_\eps, 1}^-$, which always takes values below $1$ due to sign reasons; notice that for the backward Cauchy problem $(\textrm{PB})_{b_\eps, 1}^-$ we have uniqueness thanks to the fact that $R(s)=\sqrt{s(2-s)}/(1-s)$ is increasing. The dependence of $y_{b_\eps, 1}^-(\alpha)$ on $b_\eps$ plays then a central role and allows to conclude the statement. Analogous considerations may be done for decreasing solutions. 

We now extend our considerations to general (possibly nonmonotone) heteroclinic traveling waves. Motivated in particular by a population dynamics perspective, we restrict our interest to planar fronts taking values between $0$ and $1$, otherwise other situations - meaning $f$ as extended by $0$ outside $[0, 1]$ - could in principle be considered; under this assumption, the profile $z \mapsto v(z)$ oscillates infinitely many times around the equilibrium $\alpha$ and will be obtained by gluing several monotone pieces of profile, alternatively increasing and decreasing.
If at $+\infty$ (or $-\infty$) the profile reaches another equilibrium ($0$ or $1$) in a monotone way, then this can be done starting the shooting procedure from such equilibrium, similarly as before. To this end, notice that if $y_{b_\eps, 1}^-$ satisfies $y_{b_\eps, 1}^-(\bar{v})=0$ for some $\bar{v} \in (0, 1)$, by \eqref{cambio1} the corresponding wave profile has a critical point at $z(\bar{v})$, namely $v'(z(\bar{v}))=0$. Observe that we can always assume $\bar{v} \neq \alpha$, since by the uniqueness the equilibrium $\alpha$ is reached in infinite time, see also the second order formulation \eqref{IIo} below. For this reason, it is $(y_{b_\eps, 1}^-)'(\bar{v}) \neq 0$ and thus $v''(z(\bar{v})) \neq 0$, meaning that $z \mapsto v'(z)$ changes sign in $z(\bar{v})$ (i.e., when the profile takes the value $\bar{v}$). Now, the subsequent piece of profile will be \emph{decreasing} in the $z$-variable; in order to determine it as above, 
one has to follow the solution of the \emph{forward} Cauchy problem $(\textrm{PB})_{-b_\eps, \bar{v}}^+$
up to its first zero $\bar{\bar{v}}$; we explicitly underline that the sign of the speed has to be changed when constructing a piece of solution of the first order problem with opposite monotonicity, see \cite[Remark 2.3]{GarSan}.
The entire profile $z \mapsto v(z)$ up to its value $\bar{\bar{v}}$ will then be obtained by juxtaposing the constructed pieces of graphs of $y_{b_\eps, q_2}^-$ and $y_{-b_\eps, \bar{v}}^+$; notice that this gives rise to a regular ($C^2$) wave solution, since the gluing procedure in correspondence of the value $y=0$ preserves the $C^1$-regularity thanks to \eqref{cambio1}; using \eqref{TWgeneral}, one then reaches the $C^2$-regularity. A brief comment about the case when no equilibria are reached monotonically at $\pm \infty$ is postponed after the forthcoming proof.

With this preliminary discussion in mind, we get the following result. 
\begin{theorem}\label{nonomon}
Fix $\eps > 0$. Then, 
\begin{itemize}
\item for every $b_\eps \in (b_\eps^*, b_\eps^+)$, there exist a nonmonotone traveling wave type solution of \eqref{TWW} connecting $\alpha$ and $1$, which is definitively increasing at $+\infty$;
\item for every $b_\eps \in (0, b_\eps^+)$, there exist a nonmonotone traveling wave type solution of \eqref{TWW} connecting $0$ and $\alpha$, which is definitively decreasing at $+\infty$.
\end{itemize}
\end{theorem}
\begin{proof}
The idea is to use the gluing procedure previously described. As for the first item, let us initially assume that $\eps > \int_0^1 f^-(s) \, ds$. In this case, we start from $y_{b_\eps, 1}^-$, 
which cannot blow up to the boundary $\{y=1\}$ due to sign reasons. Indeed, if it did it, its derivative therein should be equal to $-\infty$ (since $y_{b_\eps, 1}^-(1)=0$), while by the differential equation
\begin{equation}\label{bb}
y'= b_\eps \frac{\sqrt{y(2-y)}}{1-y} - g_\eps(v)
\end{equation}
satisfied by $y_{b_\eps, 1}^-$ it can only be equal to $+\infty$, being $b_\eps > b_\eps^* > 0$. Now, the solution of \eqref{bb} fulfilling $y(0)=0=y(1)$, obtained for $b=b_\eps^*$ and corresponding to the monotone front connecting $0$ and $1$ found in Proposition \ref{monotoni}, is a strict supersolution for $\textrm{(PB)}_{b_\eps, 1}^-$, while $0$ is a strict subsolution as long as $v \geq \alpha$. Hence, $y_{b_\eps, 1}^-$ is positive and vanishes at some point $v_1$; since it has to be $(y_{b_\eps, 1}^-)'(v_1) \geq 0$ (remember that $y_{b_\eps, 1}^-$ is positive so far), from \eqref{bb} we deduce that $v_1 < \alpha$, in view of the sign of $g_\eps$. We now construct a decreasing piece of planar wave profile by solving, in the first order formulation, the forward Cauchy problem $\textrm{(PB)}_{-b_\eps, v_1}^+$, as mentioned in \cite[Remark 2.3]{GarSan};
noticing that $y_{b_\eps, 1}^-$ is here a strict supersolution, $y_{-b_\eps, v_1}^+$ is going to vanish in a point $v_2$ which belongs to the interval $(\alpha, 1)$ in view of the sign of $g_\eps$. We now iterate such a procedure, alternatively shooting backward and forward from the zeros of the constructed pieces of solution (the subsequent one would be $v_2$) and using as strict supersolutions the pieces of solutions constructed in the immediately preceding step (the subsequent one would be $y_{-b_\eps, v_1}^+$, which is a strict supersolution for $\textrm{(PB)}_{b_\eps, v_2}^-$). By construction, it is clear that the sequences of zeros $v_{2j}$, $v_{2j+1}$ where the profile changes its monotonicity are monotone (respectively, decreasing and increasing), so they both converge to the equilibrium $\alpha$. 
In case $\eps \leq \int_0^1 f^-(s) \, ds$, the argument would work all the same: indeed, here $y_{b_\eps, 1}^-$ would vanish in a certain $0 < v_1 < \alpha$ because otherwise its graph would cross the one of the solution of 
$\textrm{(PB)}_{b_\eps, 0}^+$,
which blows up to the boundary $\{y=1\}$. By the uniqueness, this would not be possible; notice that the conclusion could be drawn similarly also in case $\int_0^1 f(s) \, ds=0$, for which $b_\eps^* = 0$. 

As for the second item, this time we start with 
$y_{-b_\eps, 0}^+$, for $0 < b_\eps < b_\eps^+$,  
noticing that the blow-up to the barrier $\{y=1\}$ cannot occur since the primitive of $-g_\eps$ taking zero value at $0$ is a supersolution which vanishes in a point of the interval $(\alpha, 1)$, being $\int_0^1 f(s) \, ds > 0$. 
We can thus proceed as for the previous item; since $y_{-b_\eps, 0}^+$ is well defined up to its first zero, it works as a supersolution for the following steps and hence blow-up is not possible at any of the subsequent iterations. The conclusion can then be obtained as before.
\end{proof}
Notice that in this case there could be high multiplicity of solutions; the closer $v(0)$ to $\alpha$ is, the more solutions are obtained. To maintain uniqueness, one should choose $v(0)$ close to $1$ in such a way that $z \mapsto v(z)$ takes the value $v(0)$ only once.
In principle, there could also be the possibility for heteroclinic traveling waves to oscillate around $\alpha$ at $+\infty$ ($-\infty$) and to have inferior and superior limit respectively equal to $0$ and $1$ at $-\infty$ ($+\infty$); in view of Proposition \ref{monotoni}, for positive speeds this could occur only if $b_\eps < b_\eps^*$. Anyway, the point is that a profile of this kind would make infinite oscillations both for $z \to +\infty$ and for $z \to -\infty$, and for this reason the corresponding pieces of solution of the first order problem have to be shot both backward and forward from each of their zeros (with opposite speed, as we have already seen). It is not difficult to see that this forces the corresponding profile $z \mapsto v(z)$ to necessarily take values also outside $[0, 1]$: at some point, the backward solution $y_{b_\eps, \hat{v}}^-$, shot from a suitable $\hat{v}$, would stay above $y_{0, 0}^+$, but then it would take positive value in $0$, thus becoming not admissible for our purposes.

Finally, we observe that for $b=b_\eps=0$ the two branches $y_{b_\eps}^+$ and $y_{b_\eps}^-$ coincide for $\eps > \bar{\eps}=\int_0^1 f^-(s) \, ds$ (here the differential equation in \eqref{ilsistema} becomes $y'=-g_\eps(v)$), so that the solution bounces regularly between $0$ and the value $v_0$ defined by $\int_0^{v_0} f(s)\, ds = 0$; for $\eps = \bar{\eps}$ such a bouncing profile remains continuous but takes infinite derivative in infinitely many points, while for lower values of $\eps$ it disappears and the only possible solution is provided by Proposition \ref{dss} (the fact that $\int_0^1 f(s) \, ds > 0$ prevents instead the appearance of regular steady profiles connecting $0$ and $1$). In view of the remark after Proposition \ref{monotoni}, nonmonotone heteroclinic traveling waves tend to disappear for $\eps \to 0$, since the critical speed for monostable-type fronts converges to $0$ for $\eps \to 0$. 
\begin{remark}
\textnormal{Some of the above conclusions could be reached also by a careful analysis of the equilibria for the second order ODE system equivalent to \eqref{TWW}:
\begin{equation}\label{IIo}
\left\{
\begin{array}{l}
\displaystyle V'=\frac{W}{\sqrt{1-W^2}} \vspace{0.1cm} \\
\displaystyle W'= b_\eps \frac{W}{\sqrt{1-W^2}}  - g_\eps(V). 
\end{array}
\right.
\end{equation}
However, the above direct study of the first order model can be implemented with little difficulty also for general operators and appears simpler than the complete analysis of the second order system. 
}
\end{remark}

\begin{remark}\label{altroc}
\textnormal{Dealing with \eqref{TW}, namely keeping $\eps$ in front of the diffusive term avoiding to divide by $\eps$, the same change of variables as before would yield the first-order two-point problem
\begin{equation}\label{IIeq}
\left\{
\begin{array}{l}
\displaystyle y'= c \frac{\sqrt{y(2\eps-y)}}{\eps-y} - f(v), \vspace{0.1cm}\\
y(0)=0, y(1)=0, \; 0 < y(v) < \eps \textrm{ for } v \in \,(0, 1),
\end{array}
\right.
\end{equation}
where 
$$
y(w)=
\eps\left(1 - \frac{1}{\sqrt{1+v'(z(v))^2}}\right).
$$
The behavior of fronts could here be deduced similarly as above. This alternative way of proceeding may actually be independently helpful: see the proof of Theorem \ref{velconv0} and Section \ref{sez3}.}
\end{remark}

So far, we have reasoned for fixed $\eps > 0$. Let us now highlight two properties on varying of $\eps$ which may be useful. 

\begin{proposition}
For $\eps > \int_0^1 f^-(s) \, ds$, the function $\eps \mapsto c^*_\eps$ is continuous and monotone decreasing.  
\end{proposition}

\begin{proof}
The continuity follows from the continuous dependence for the differential equation \eqref{ilsistema} in the case of regular solutions; the monotonicity follows instead from standard lower and upper solution arguments, see for instance \cite{GarSan}. 
\end{proof}

For $\eps$ approaching $\bar{\eps}=\int_0^1 f^-(s) \, ds$, we now show that the speed of the critical front connecting $0$ and $1$ goes to $0$, and thus the occurrence of discontinuous steady states appears naturally justified. 

\begin{theorem}\label{velconv0}
It holds 
$$
\lim_{\eps \to \bar{\eps}} c^*_\eps = 0. 
$$
\end{theorem}

\begin{proof}
Here it is more convenient to use the first-order model in the form \eqref{IIeq}.
Assume by contradiction that the statement is not true; since $c^*_\eps$ is positive and decreasing as a function of $\eps$, this means that there exists $\bar{c} > 0$ with $c^*_\eps \searrow \bar{c}$ for $\eps \to \bar{\eps}$. Fix $\eps > \bar{\eps}$ and let $\alpha_0$ be defined by $\int_0^{\alpha_0} f^-(s) \, ds = \bar{\eps}/2$. 
Denoting by $y_\eps$ the solution of 
\begin{equation}\label{pyeps}
\left\{
\begin{array}{l}
\displaystyle y'= c^*_\eps \frac{\sqrt{y(2\eps-y)}}{\eps-y} - f(v) \vspace{0.1cm}\\
y(0)=0
\end{array}
\right.
\end{equation}
(unique by virtue of \cite{BonSan}, thanks to the negative sign of $f$ in a neighborhood of $0$) and observing that $s_\eps(v)=\int_0^v -f(s) \, ds$ is a lower solution for \eqref{pyeps} in the interval $[0, \alpha]$, we can write $c^*_\eps= \bar{c}+h(\eps)$ with $h(\eps) > 0$ and infer that 
\begin{align*}
y_\eps'(v) & \geq \bar{c} \frac{\sqrt{y_\eps(v)(2\eps-y_\eps(v))}}{\eps-y_\eps(v)} - f(v)  \\
& \geq \bar{c} \frac{\sqrt{(\bar{\eps}/2)(2\eps-\bar{\eps}/{2})}}{\eps-\bar{\eps}/2} - f(v) \quad \,  \textrm{ since } \; y_\eps(v) \geq s_\eps(v) \geq \frac{\bar{\eps}}{2} \textrm{ for } v \in [\alpha_0, \alpha].
\end{align*}
Integrating between $\alpha_0$ and $\alpha$ this last relation yields 
$$
y_\eps(\alpha) - y_\eps(\alpha_0) \geq \bar{c} \frac{\sqrt{(\bar{\eps}/2)(2\eps-\bar{\eps}/{2})}}{\eps-\bar{\eps}/2} (\alpha-\alpha_0) + \frac{\bar{\eps}}{2}, 
$$
giving 
$$
y_\eps(\alpha) \geq \bar{\eps} + \bar{c} \frac{\sqrt{(\bar{\eps}/2)(2\eps-\bar{\eps}/{2})}}{\eps-\bar{\eps}/2} (\alpha-\alpha_0). 
$$
For $\eps \searrow \bar{\eps}$, the second summand in the right-hand side of the above inequality is a constant which goes to $\bar{c}\sqrt{3}$, while at the same time it has to be $y_\eps(\alpha) < \eps$, otherwise $y_\eps$ would not correspond to a regular front. This is a contradiction for $\eps$ sufficiently close to $\bar{\eps}$. 
\end{proof}

\section{The limit for $\eps \to 0$}\label{sez3}

We are here interested in some convergence properties of the traveling fronts when passing to the limit for $\eps \to 0$. Here it appears convenient to stick to the change of variables mentioned in Remark \ref{altroc}, namely we do not divide by $\eps$ the original equation. 
We will analyze the convergence from two different points of view, that is, fixing the speed (in case it is admissible for every small $\eps$) or focusing on the critical speeds and examining the convergence of the associated fronts. For simplicity, we will perform the study only in the case \eqref{lascelta}, briefly remarking the changes to be done for a general strongly saturating diffusion.

\subsection{Convergence at fixed speed}

In the first case, inspired by \cite{HilKim}, we recall that an \emph{inviscid traveling wave} is a planar traveling wave-type solution of the problem without diffusion, that is, $u(x, t)=v(x \cdot e +ct)$ solving 
\begin{equation}\label{inviscido}
u_t=f(u) \quad \Rightarrow \quad cv'=f(v).
\end{equation}
As mentioned in \cite[Remark 2]{HilKim}, for the bistable case no fronts of this type connecting $0$ and $1$ appear, while there always exists an inviscid traveling front connecting $\alpha$ and $1$ (and, symmetrically, connecting $0$ and $\alpha$). Such a front can be recovered simply by solving 
the ODE in \eqref{inviscido} by separation of variables, noticing that the integral $\int_{\alpha}^1 dv/f(v)$ diverges since $f'(\alpha) > 0$ by assumption.
It can thus be proved that for every speed $c \geq 0$ there exists an inviscid traveling front $\mathcal{V}_c$ connecting $\alpha$ and $1$ with speed $c$ (this is independent of $\eps$, since $\eps$ does not appear in equation \eqref{inviscido}). On the other hand, the fact that $c^+_\eps \to 0$ for $\eps \to 0$ implies that for fixed speed $c > 0$ it suffices to take $\eps$ sufficiently small in order to find a regular front connecting $\alpha$ and $1$ having speed $c$; we denote such a front by $v_{\eps, c}$. We assume without loss of generality that $v_{\eps, c}(0)=(\alpha+1)/2$ for every $\eps >0$. We now have the following.  
\begin{theorem}\label{convinv}
Let $c > 0$ be fixed. Then, $v_{\eps, c} \to \mathcal{V}_c$ for $\eps \to 0$. 
\end{theorem}
\begin{proof}
First, we observe that since $\alpha \leq v_{\eps, c}(z) \leq 1$ for every $z \in \mathbb{R}$, the set $\{v_{\eps, c}\}_\eps$ is a bounded subset of $L^\infty(\mathbb{R})$; by using a diagonal procedure, then, we have the existence of a function $v_c$ such that $v_{\eps, c}(z) \to v_c(z)$ almost everywhere in $\mathbb{R}$. 
Moreover, multiplying \eqref{TW} by $v_{\eps, c}'$ and integrating on $\mathbb{R}$ yields 
\begin{equation}\label{integrale1}
\eps \int_{\mathbb{R}} \frac{v_{\eps, c}''(z)}{\sqrt{1+v_{\eps, c}'(z)^2}^3} v_{\eps, c}'(z) \, dz - c \int_{\mathbb{R}} v_{\eps, c}'(z)^2 \, dz + \int_{\mathbb{R}} f(v_{\eps, c}(z)) v_{\eps, c}'(z) \, dz = 0  
\end{equation}
and since 
$$
\int_{\mathbb{R}} \frac{v_{\eps, c}''(z)}{\sqrt{1+v_{\eps, c}'(z)^2}^3} v_{\eps, c}'(z) \, dz = -\frac{1}{\sqrt{1+v_{\eps, c}'(z)^2}} \Bigg\vert_{-\infty}^{+\infty} = 0,
$$
we infer that 
$$
\Vert v_{\eps, c}' \Vert_{L^2(\mathbb{R})} = \frac{F(1)-F(\alpha)}{c},
$$
where $F$ is a primitive of $f$. 
Thus $v_{\eps, c} \in H^1_{\textnormal{loc}}(\mathbb{R})$; by the compact Sobolev embedding into continuous functions, 
we have that $v_{\eps, c} \to v_c$ locally uniformly. Notice that $v_c$ is H\"older continuous (again by the Sobolev embeddings) and is increasing since $v_{\eps, c}$ is increasing for every $\eps$. Using \cite[Lemma 2.4]{Diek}, we can now deduce that $v_{\eps, c} \to v_c$ uniformly in $\mathbb{R}$, since $v_{\eps, c}$ is increasing, $v_c$ is continuous and there is pointwise convergence of $v_{\eps, c}$ to $v_c$ also for $z \to \pm\infty$. 
Multiplying \eqref{TW} by a test function $\psi \in C_c^\infty(\mathbb{R})$ and integrating by parts, one now has 
$$
-\eps \int_{\mathbb{R}} \frac{v_{\eps, c}'(z)}{\sqrt{1+(v_{\eps, c}'(z))^2}} \psi'(z) \, dz + c \int_{\mathbb{R}} v_{\eps, c}(z) \psi'(z) \, dz + \int_{\mathbb{R}} f(v_{\eps, c}(z)) \psi(z) \, dz =0.
$$
Since $\frac{v_{\eps, c}'(z)}{\sqrt{1+(v_{\eps, c}'(z))^2}} \leq 1$ and $\alpha \leq v_{\eps, c}(z) \leq 1$, the Lebesgue dominated convergence theorem ensures that 
$$
-\eps \int_{\mathbb{R}} \frac{v_c'(z)}{\sqrt{1+(v_c'(z))^2}} \psi'(z) \, dz + c\int_{\mathbb{R}} v_c(z) \psi'(z) \, dz  + \int_{\mathbb{R}} f(v_c(z)) \psi(z) \, dz = 0.
$$
For $\eps \to 0$, this says that $v_c$ is a weak solution of 
$
cv_c' - f(v_c) = 0.
$
We thus conclude similarly as in \cite[Theorem 1]{HilKim}: 
since $v_c$ is H\"older continuous and cannot be constant in view of the fact that $v_{c}(0)=(\alpha+1)/2$, by the uniqueness one has $\alpha < v_c < 1$, so that $v_c' > 0$ and thus $v_c$ coincides with the inviscid front of speed $c$.  
\end{proof}

\begin{remark}
\textnormal{In case of a general strongly saturating diffusion one can proceed analogously, splitting the integration domain in the first integral appearing in \eqref{integrale1} into the two domains $\{v''_{\eps, c} \gtrless 0\}$, and exploiting \eqref{diff} to infer that the two obtained integrals compensate one for the other. Thus, the first term in \eqref{integrale1} disappears and the rest of the argument works in the same way. 
The first part of the above proof also shows that a traveling wave having nonzero speed has necessarily to be continuous, since its derivative belongs to $L^2$ (see also \cite{KurRos}).}
\end{remark}

\subsection{Convergence of the critical fronts connecting $\alpha$ and $1$}

As for the critical fronts connecting $\alpha$ and $1$, which under our hypotheses appear for $c_\eps^+=2\sqrt{\eps}\sqrt{f'(\alpha)}$, we denote them by $V^+_\eps=v_{\eps, c^+_\eps}$. We also set 
$$
H_\alpha(z)=\left\{
\begin{array}{ll}
\alpha & z < 0 \\
(\alpha+1)/2 & z = 0 \\ 
1 & z > 0.
\end{array}
\right.
$$
Then, 
we have the following.  
\begin{theorem}\label{conv2}
For every $z \in \mathbb{R}$, it holds 
$$
V_\eps^+(z) \to H_\alpha(z),
$$
where the convergence is uniform in $\mathbb{R} \setminus \mathcal{I}_0$, $\mathcal{I}_0$ being an arbitrary neighborhood of the origin. Moreover, 
$$
(V_\eps^+)' \to (1-\alpha) \delta_0 \quad \textrm{ in } \mathcal{D}'(\mathbb{R}),
$$
where $\delta_0$ denotes the Dirac delta distribution concentrated at $0$. 
\end{theorem}
\begin{proof}
Recalling that $\alpha \leq V_\eps^+(z) \leq 1$ for every $z \in \mathbb{R}$, using a diagonal procedure as in the proof of Theorem \ref{convinv} we have that there exists $V^+$ for which $V_\eps^+ \to V^+$ almost everywhere. 
Multiplying \eqref{TW} by $\psi \in C_c^\infty(\mathbb{R})$ and integrating by parts, we then infer that 
$$
-\eps \int_{\mathbb{R}} \frac{(V_\eps^+)'(z)}{\sqrt{1+((V_\eps^+)'(z))^2}} \psi'(z) \, dz +2\sqrt{\eps} \sqrt{f'(\alpha)} \int_{\mathbb{R}} V_\eps^+(z) \psi'(z) \, dz + \int_{\mathbb{R}} f(V_\eps^+(z)) \psi(z) \, dz = 0. 
$$
Passing to the limit for $\eps \to 0$, we observe that the first two summands converge to $0$ because $\frac{(V_\eps^+)'(z)}{\sqrt{1+((V_\eps^+)'(z))^2}} \leq 1$ and $\alpha \leq V_\eps^+(z) \leq 1$ for every $z \in \mathbb{R}$. It follows that 
$$
\lim_{\eps \to 0} \int_{\mathbb{R}} f(V_\eps^+(z)) \psi(z) \, dz = 0,
$$
for every $\psi \in C_c^\infty(\mathbb{R})$. 
Since $f$ is bounded and $f(V_\eps^+(z)) \to f(V^+(z))$ for almost every $z \in \mathbb{R}$, we deduce that $f(V^+(z)) = 0$ for almost every $z \in \mathbb{R}$. 
Using the monotonicity of $V_\eps^+$ and the fact that $V^+(0)=\frac{\alpha+1}{2}$, we deduce that 
$$
V^+(z) = \alpha \textrm{ for almost every } z < 0, \qquad V^+(z) = 1 \textrm{ for almost every } z > 0. 
$$
The full pointwise convergence follows now from the fact that $z \mapsto V_\eps^+(z)$ is monotone for every $\eps$, together with the comparison theorem for limits. The uniform convergence outside a neighborhood of $0$ follows instead from \cite[Lemma 2.4]{Diek}, similarly as in the proof of Theorem \ref{convinv} (notice that the limit function has to be continuous in order to apply such a lemma). 
It remains to prove that $(V_\eps^+)' \to (1-\alpha) \delta_0$ in distributional sense, namely 
$$
-\lim_{\eps \to 0} \langle V_\eps^+, \psi' \rangle = (1-\alpha) \psi(0) \quad \textrm{ for every } \psi \in C_c^{\infty}(\mathbb{R}),
$$
where the first duality can be meant in integral sense since $V_\eps^+ \in L^1_{\textnormal{loc}}(\mathbb{R})$. 
Using the Lebesgue dominated convergence theorem we now have 
$$
-\lim_{\eps \to 0} \int_\mathbb{R} V_\eps^+(z) \psi'(z) \, dz = -\int_{\textnormal{supp}\, \psi} H_\alpha(z) \psi'(z) =(1-\alpha) \psi(0),
$$ 
whence the thesis.
\end{proof}
In case of a general strongly saturating diffusion like the ones considered in \eqref{laPDE}, the fact that $P$ is bounded ensures that the argument in the previous proof works all the same. 

In Figure \ref{figura}, we illustrate Theorem \ref{conv2} for $P$ as in \eqref{lascelta},
\begin{equation}\label{sceltaf}
f(s)=s(1-s)(s-0.4)
\end{equation}
and different values of $\eps$: it is quite immediate to see that, the smaller $\eps$, the steeper the corresponding traveling front (on the left). On the right, we depict the shape of the corresponding solution of \eqref{ilsistema} (with $q_1=\alpha$, $q_2=1$). Notice that \eqref{ilsistema} provides a very useful way of determining the initial value of the derivative for the traveling front: solving for $v'(z(v))$ in \eqref{cambio1}, in correspondence of the point $v=v_\eps(0)=(\alpha+1)/2$, we can find $v_\eps'(0)$ and thus shoot the solution properly in our simulations. Due to the high sensitivity of the numerical response to small variations of the initial data and to the objective difficulty of finding a solution defined on the whole real line, we here notice another advantage of the above change of variables. 

\begin{figure}[ht!]
\includegraphics[scale=0.49]{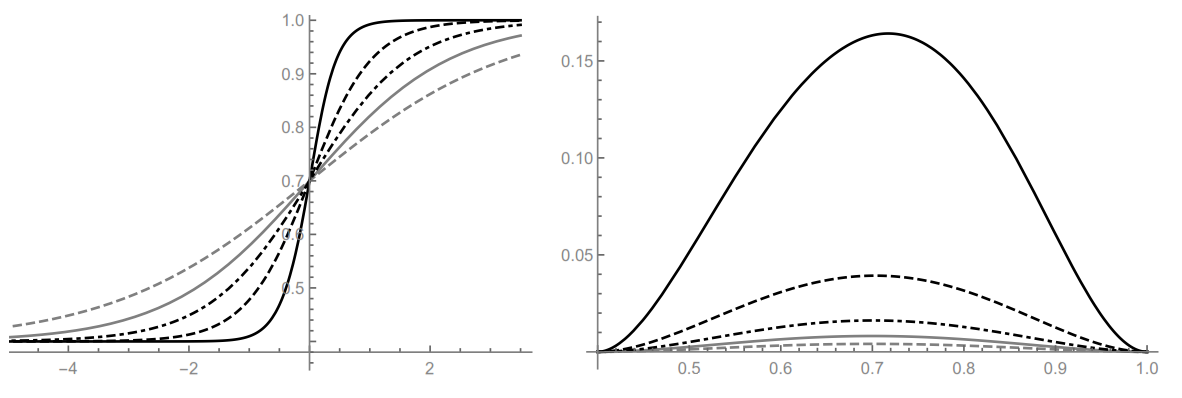}
\caption{For $f$ as in \eqref{sceltaf} and $\eps=0.5$ (gray, dashed), $\eps=0.25$ (gray), $\eps=0.125$ (black, dot-dashed), $\eps=0.05$ (black, dashed), $\eps=0.01$ (black), we depict: on the left, the critical traveling fronts connecting $\alpha$ and $1$ and solving \eqref{TW}, on the right the corresponding solution $y$ of \eqref{ilsistema}.}\label{figura}
\end{figure}

\subsection{Convergence of the critical fronts connecting $0$ and $1$}

We now take into account the unique (up to translations) increasing front which connects $0$ and $1$, which henceforth we denote by $V_\eps^*=v_{\eps, c^*_\eps}$. We have seen that, for $\eps > \int_0^1 f^-(s) \, ds$, this is a regular front, ``normalized'' in such a way that $V_\eps^*(0)=1/2$; on the contrary, for $\eps \leq \int_0^1 f^-(s) \, ds$, $V_\eps^*$ is a discontinuous steady state, and in this case we have chosen to ``normalize'' it in such a way that 
it is discontinuous at $z=0$ and $V_\eps^*(0)=1/2$. 
For $\eps \to \bar{\eps}$, we have seen that the appearance of discontinuous steady states agrees with Theorem \ref{velconv0}, and it is even more justified if one writes the definition of weak $BV_{\textnormal{loc}}$-solution as in \cite[formula (6)]{BonObeOma} for (the regular solution) $V_\eps^*$ and lets $\eps \to \bar{\eps}$, precisely in view of the fact that $c_\eps^* \to 0$ for $\eps \to \bar{\eps}$ (see also \cite[Example 1.1 and Remark 1.1]{BonObeOma}). 

We now want to analyze the behavior of $V_\eps^*$ for $\eps \to 0$. 
Denoting by $H_0$ the Heaviside function 
$$
H_0(z)=\left\{
\begin{array}{ll}
0 & z < 0 \vspace{0.1cm} \\
1/2 & z = 0 \vspace{0.1cm} \\
1 & z > 0, 
\end{array}
\right.
$$ 
we have the following. 
\begin{theorem}\label{convAE}
For every $z \in \mathbb{R}$, it holds
$$
V_\eps^*(z) \to H_0(z),
$$
where the convergence is uniform in $\mathbb{R} \setminus \mathcal{I}_0$, $\mathcal{I}_0$ being an arbitrary neighborhood of the origin. Moreover, 
$$
(V_\eps^*)' \to \delta_0 \quad \textrm{ in } \mathcal{D}'(\mathbb{R}).
$$
\end{theorem}
\begin{proof}
Instead of using the same method as for the proof of Theorem \ref{conv2}, we show how the first-order reduction can here be effective also in a convergence argument. For $\eps \leq \int_0^1 f^-(s) \, ds$ we notice that, through the change of variables of the previous section, the function $y$ (defined as in \eqref{cambio1}) associated with $V_\eps^*$ 
obeys the equation
\begin{equation}\label{ausiliaria}
\left\{
\begin{array}{l}
\displaystyle y'= - f(v), \vspace{0.1cm}\\
y(0)=0, y(1)=0, \; 0 < y(v) < \eps,
\end{array}
\right.
\end{equation}
namely \eqref{IIeq} with $c=0$. The solution of such a problem is given by 
$$
y_\eps(v)=\left\{
\begin{array}{ll}
F^-(v) & v \in [0, v_\eps^-) \vspace{0.1cm} \\ 
F^+(v) & v \in (v_\eps^+,1], 
\end{array}
\right.
$$
where $F^-(s)=-\int_0^s f(\tau) \, d\tau$, $F^+(s)=\int_s^1 f(\tau) \, d\tau$ and $v_\eps^-, v_\eps^+$ are implicitly defined by
\begin{equation}\label{leimplicite} 
F^-(v_{\eps^-})=\eps=F^+(v_{\eps^+}); 
\end{equation}
notice that $v_\eps^- < \alpha < v_\eps^+$. 
Since $v$ represents the value of the front profile, this means that the corresponding front $V_\eps^*$ is defined, $C^2$ and increasing both on the left and on the right of $0$ (where we have placed its discontinuity) and that 
$$
\lim_{z \to 0^-} V_\eps^*(z)=v_\eps^-, \qquad \lim_{z \to 0^+} V_\eps^*(z)=v_\eps^+.
$$ 
On decreasing of $\eps$, in view of \eqref{leimplicite} and taking into account the sign of $f$, we have that $v_\eps^-$ decreases and $v_\eps^+$ increases.  
As $f$ is strictly negative in a neighborhood of $0$ and strictly positive in a neighborhood of $1$, it is now clear that $v_\eps^- \to 0$, $v_\eps^+ \to 1$, otherwise \eqref{leimplicite} would be violated in the limit for $\eps \to 0$. However, this means that the limit $\bar{V}(z):=\lim_{\eps \to 0} V_\eps^*(z)$ (which can be constructed pointwise similarly as in the previous proofs) will be such that 
$$
\lim_{z \to 0^-} \bar{V}(z)=0, \qquad \lim_{z \to 0^+} \bar{V}(z)=1.
$$
Since $\bar{V}$ is the limit of positive increasing functions on $(-\infty,0)$ and on $(0, +\infty)$, this means that $\bar{V}=H_0$ almost everywhere. The pointwise convergence and the uniform convergence outside a neighborhood of $0$ follow from the same argument as for Theorem \ref{conv2} (notice that $V_\eps^*(0)=1/2$ by the previous positions). 

It remains to prove that $(V_\eps^*)' \to \delta_0$ in distributional sense; this follows similarly as in the proof of Theorem \ref{conv2} since
$$
-\lim_{\eps \to 0} \int_\mathbb{R} V_\eps^*(z) \psi'(z) \, dz = -\lim_{\eps \to 0} \int_{\textnormal{supp}\, \psi} V_\eps^*(z) \psi'(z) = -\int_{\textnormal{supp}\, \psi} H_0(z) \psi'(z) = \psi(0).
$$ 
\end{proof}
The argument used throughout the proof works as well for general strongly saturating diffusions, up to possibly replacing the constant $\eps$ bounding $y$ in \eqref{ausiliaria} with another constant depending continuously on $\eps$; of course, the differential equation therein remains instead the same, since the speed of the discontinuous stationary waves is equal to $0$ (and thus the part coming from the second-order operator disappears).  
\newline
In Figure \ref{figura2}, we illustrate Theorem \ref{convAE} for $P$ as in \eqref{lascelta} and $f$ as in \eqref{sceltaf}, for different values of $\eps$: notice that here $\bar{\eps}=\int_0^1 f^-(s) \, ds \approx 0.00853$, so for values of $\eps$ converging to $\bar{\eps}$ from above, the regular fronts (on the left) keep becoming steeper and already for $\eps=0.01$ (gray, dashed), corresponding to $c^* \approx 0.0006326$, the derivative of the profile in $0$ is almost infinite. On decreasing of $\eps$ starting from the value $\bar{\eps}$, the solutions are discontinuous steady states, and this is well seen both in the left picture and in the right one, where we depict the shape of the associated solution $y$ of \eqref{ilsistema} (with $q_1=0$, $q_2=1$). Again, solving \eqref{ilsistema} is the only way to know what is the initial derivative with which we can shoot the solution in order to reconstruct the wave profile. 

\begin{figure}[ht!]
\includegraphics[scale=0.47]{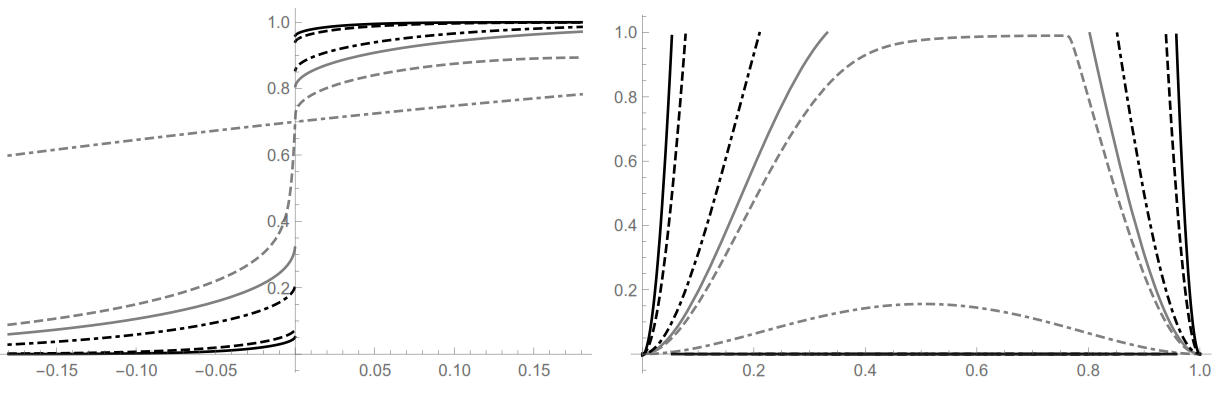}
\caption{For $f$ as in \eqref{sceltaf} and $\eps=0.1$ (gray, dot-dashed), $\eps=0.01$ (gray, dashed), $\eps=0.008$ (gray), $\eps=0.005$ (black, dot-dashed), $\eps=0.001$ (black, dashed), $\eps=0.0005$ (black) we depict: on the left, the critical traveling fronts connecting $0$ and $1$ and solving \eqref{TW}, on the right the corresponding solution $y$ of \eqref{ilsistema}.}\label{figura2}
\end{figure}

\begin{remark}
\textnormal{It is likely that a similar approach to the one developed in \cite{GilKer} could provide sharper results than the ones presented here, possibly allowing to relax the regularity assumptions on the reaction term. However, proceeding similarly as in \cite{GilKer}, the integral equation to be studied could here be obtained from \eqref{IIo} and would read as
$$
x(t)=f(t) + \int_0^t \frac{g(s) \sqrt{1-x(s)^2}}{x(s)} \, ds.
$$
Though seeming not particularly more complicated than the one studied in \cite{GilKer}, such equation would probably require a number of ad-hoc preliminary results, the ones in \cite{Gil}  not being directly applicable. For this reason, we have preferred to stick to a more elementary and direct approach, focusing on the nature of the results for the PDE model \eqref{laPDE} rather than on their optimal statement}.
\end{remark}

\begin{remark}\label{remMink}
\textnormal{The (positive definite Euclidean) curvature operator is usually mentioned together with its analogous in Lorentz spaces, the so-called \emph{Minkowski curvature operator}, given by choosing
$$
P(s)=\frac{s}{\sqrt{1-s^2}}. 
$$ 
It is natural to wonder what part of the present investigation can be extended to this qualitatively different case. On the one hand, the existence of fronts for such a kind of operator was analyzed in \cite{CoeSan}, where it was shown that for monostable reaction terms a similar picture as for the linear case appears, though the control required on the growth of the reaction is therein slightly stronger. More precisely, in \cite{CoeSan} it is assumed that 
\begin{equation}\label{controllomink}
f(s) \leq \frac{M(s-\alpha)}{\sqrt{1+M^2 (s-\alpha)^2}} \quad \textrm{ for every } s \in [\alpha, 1]
\end{equation}
in order to state that the critical speed $c^+$ for planar fronts connecting $\alpha$ and $1$ for the equation  
$$
u_t= \textnormal{div}\,\left(\frac{\nabla u}{\sqrt{1-\vert \nabla u \vert^2}}\right) + f(u) 
$$
satisfies the upper bound $c^+ \leq 2\sqrt{M}$. The problem here is that, given $f$ for which \eqref{controllomink} holds for some $M$, it is not true that $g(s)=f(s)/\eps$ fulfills \eqref{controllomink} for $M'=M/\eps$ (compare with Remark \ref{stimaM} above); in fact, if $f$ reaches a maximum value $K > 0$ inside the interval $(0, 1)$, it suffices that $K/\eps > 1$ in order for \eqref{controllomink} to fail at least in one point of the interval, regardless of $M$. It is thus not clear what portrait to be expected for $\eps \to 0$; one may wonder whether 
heteroclinic fronts between $\alpha$ and $1$ still exist, or whether any limit procedure necessarily passes through nonmonotone traveling wave solutions. The bistable case could possibly be even more complicated. These issues will be the object of a future investigation. }

\end{remark}


\small

\begin{thebibliography}{99}
\bibitem{AriCamMar}{M. Arias, J. Campos and C. Marcelli, \emph{Fastness and continuous dependence in front propagation in Fisher-KPP equations}, Discrete Contin. Dyn. Syst. Ser. B \textbf{11} (2009), 11--30.}
\bibitem{AroWei}{D. G. Aronson and H. F. Weinberger, \emph{Multidimensional nonlinear diffusion arising in population genetics}, Adv. Math.
\textbf{30} (1976), 33--76.}
\bibitem{BonObe}{D. Bonheure and F. Obersnel, \emph{Optimal profiles in a phase-transition model with saturating flux}, Nonlinear Anal. \textbf{125} (2015), 334--357.}
\bibitem{BonObeOma}{D. Bonheure, F. Obersnel and P. Omari, \emph{Heteroclinic solutions of the prescribed curvature equation with a double-well potential}, Differential Integral Equations \textbf{26} (2013), 1411--1428.}
\bibitem{BonSan} D. Bonheure and L. Sanchez, {\it Heteroclinic orbits for some classes of second and fourth order differential equations}, 
In:
A. Canada, P. Dr\'abek, A. Fonda eds., Handbook of Differential Equations: Ordinary Differential
Equations, vol. 3,
Elsevier, Amsterdam, 2006, 103--202.
\bibitem{BurGri}{M. Burns and M. Grinfeld, \emph{Steady state solutions of a bi-stable quasi-linear equation with saturating flux}, European J. Appl. Math. \textbf{22} (2011), 317--331.}
\bibitem{CoeSan}{I. Coelho and L. Sanchez, \emph{Travelling wave profiles in some models with nonlinear diffusion}, Appl. Math. Comp. \textbf{235} (2014), 469--481.} 
\bibitem{Cro}{E.C.M. Crooks, \emph{Front profiles in the vanishing-diffusion limit for monostable reaction-diffusion-convection equations}, Differential Integral Equations \textbf{23} (2010), 495--512.}
\bibitem{CroMas}{E.C.M. Crooks and C. Mascia, \emph{Front speeds in the vanishing diffusion limit for reaction-diffusion-convection equations}, Differential Integral Equations \textbf{20} (2007), 499--514.}
\bibitem{Diek}{O. Diekmann, \emph{Limiting behaviour in an epidemic model}, Nonlinear Anal. \textbf{1} (1976/77), 459--470.}
\bibitem{FifM}{P. C. Fife and J. B. Mc Leod, \emph{The approach of solutions of nonlinear diffusion equations to travelling front
solutions}, Arch.
Ration. Mech. Anal. \textbf{65} (1977), 335--361.}
\bibitem{Fis}{R. A. Fisher, \emph{The wave of advance of advantageous genus}, Ann. Eugenics \textbf{7} (1937), 355--369.}
\bibitem{GarSan}{M. Garrione and L. Sanchez, \emph{Monotone traveling waves for reaction-diffusion equations involving the curvature operator}, Bound. Value Probl. \textbf{2015:45}, 1--31.}
\bibitem{GarStr}{M. Garrione and M. Strani, \emph{Heteroclinic traveling fronts for reaction-convection-diffusion equations with a saturating diffusive term}, to appear in Indiana Univ. Math. J.}
\bibitem{Gil}{B.H. Gilding, \emph{A singular nonlinear Volterra integral equation}, J. Integral Eq. Appl. \textbf{5} (1993), 465--502.}
\bibitem{GilKer}{B.H. Gilding and R. Kersner, Travelling waves in Nonlinear Diffusion-Convection Reaction, Birkh\"auser, Basel, 2004.}
\bibitem{GooKurRos}{J. Goodman, A. Kurganov and P. Rosenau, \emph{Breakdown in Burgers-type equatios with saturating dissipation fluxes}, Nonlinearity \textbf{12} (1999), 247--268.}
\bibitem{HilKim}{D. Hilhorst and Y-J. Kim, \emph{Diffusive and inviscid traveling waves of the Fisher equation and nonuniqueness of wave speed}, Appl. Math. Lett. \textbf{60} (2016), 28--35.}
\bibitem{KolPetPis}{A.N. Kolmogorov, I.G. Petrovsky and N.S. Piskunov, \emph{\'Etude de l'\'equation de la diffusion avec
croissance de la quantit\'e de mati\`ere et son application \`a un probl\`eme biologique}, Mosc. Univ. Math. Bull. \textbf{1} (1937),
1--25.}
\bibitem{KurRos}{A. Kurganov and P. Rosenau, \emph{On reaction processes with saturating diffusion}, Nonlinearity \textbf{19} (2006), 171--193.} 
\bibitem{LopOmaRiv}{J. L\'opez-G\'omez, P. Omari and S. Rivetti, \emph{Bifurcation of positive solutions for a one-dimensional indefinite quasilinear Neumann problem}, Nonlinear Anal. \textbf{155} (2017), 1--51.}
\bibitem{Lut}
{R.L. Luther, \emph{R\"aumliche Fortpflanzung Chemischer Reaktinen}, Z. f\"ur Elektrochemie und Angew., Physikalische Chemie \textbf{12} (1906), 506--600.} 
\bibitem{MalMar}{L. Malaguti and C. Marcelli, \emph{Travelling wavefronts in reaction-diffusion equations with convection effects and non-regular
terms}, Math. Nachr. \textbf{242} (2002), 148--164.}
\bibitem{MalMarMat}{L. Malaguti, C. Marcelli and S. Matucci, \emph{Continuous dependence in front propagation of convective reaction-diffusion equations}, 
Commun. Pure Appl. Anal. \textbf{9} (2010), 1083--1098.}
\bibitem{MatMer}{A. Mata\v{s} and J. Merker, \emph{The limit of vanishing viscosity for doubly nonlinear parabolic equations}, Electron. J. Qual. Theory Diff. Eq. \textbf{8} (2014), 1--14.}
\bibitem{Ros}{P. Rosenau, \emph{Free-energy functionals at the high-gradient limit}, Phys. Rev. A
\textbf{41} (1990),
2227--2230.
}
\end{thebibliography}
\end{document}